\documentclass[journal]{IEEEtran}
\usepackage{amsfonts}

\usepackage{amsmath}
\usepackage{amssymb}
\usepackage{xcolor}
\usepackage{graphicx}
\usepackage{cite}
\usepackage{array}
\usepackage{subfigure}
\usepackage{booktabs}
\usepackage{enumerate}
\usepackage{color}
\usepackage{algpseudocode}
\usepackage{epstopdf}
\usepackage{ntheorem}
\usepackage{pifont}
\usepackage{graphicx}
\usepackage{setspace}

\ifCLASSINFOpdf
\else
\fi
\newtheorem{theorem}{Theorem}

\newtheorem{assumption}{Assumption}

\newtheorem{proposition}{Proposition}
\newtheorem{remark}{Remark}

\newtheorem{definition}{Definition}
\newenvironment{proof}[1][Proof]{\noindent \textbf{#1.} }{\  \rule{0.5em}{0.5em}}
\ifCLASSINFOpdf \else \fi
\allowdisplaybreaks[4]

\begin{document}
\begin{spacing}{1.0}
\title{Fixed-time stabilization of systems with nilpotent matrices in the presence of large unknown delays}
\author{Kang-Kang Zhang, Emilia Fridman, \textit{Fellow, IEEE}, and Pengfei Wang
\thanks{%
This work was supported in part by the Israel Science Foundation
under Grant 446/24, in part by the Chana and Heinrich Manderman
Chair at Tel Aviv University, and in part by the Azrieli International Postdoctoral Fellowship. \textit{(Corresponding author: Pengfei Wang)}}
\thanks{The authors are with the School of Electrical and Computer Engineering, Tel Aviv University, 6997801, Israel (e-mail: kangkang\_kkz@163.com; emilia@tauex.tau.ac.il; wangpengfei1156@hotmail.com)}
}
\maketitle

\begin{abstract}
In this paper, we study state-feedback fixed-time stabilization of continuous-time perturbed linear  systems with nilpotent matrices in the presence of unknown large constant input/measurement delays with known bounds. By using an additional artificial delay, we design a feedback that for a precisely known linear system with a nilpotent matrix drives the state to zero in fixed time (which is larger than the bound of the unknown delay) and achieves input-to-state stability with a large exponential decay rate in the presence of additive disturbances and small system matrix perturbations and nonlinearities. For perturbed single integrators, fixed-time stabilization is achieved in the case of fast-varying unknown measurement delays with known bounds. The results are extended to discrete-time systems. Numerical simulations illustrate the efficiency of the results.
\end{abstract}

\markboth{Paper}{Shell
\MakeLowercase{\textit{et al.}}: Bare Demo of IEEEtran.cls for IEEE Journals}

\IEEEpeerreviewmaketitle

\section{Introduction}

Due to the inevitable presence of time delays in practical physical systems, control problems for systems with time delays have attracted significant attention over the past several decades \cite{emilia14book,Gu03book,Michels07book,zhou12auto}. As a result, a considerable amount of research has been devoted to addressing systems with large delays. Various predictor methods handle large delays \cite{Krstic09book,emilia25survey}. However, conventional predictor feedback usually requires the delay to be exactly known \cite{Art03tac,Krstic10tac}, while partial differential equation (PDE) delay-adaptive control methods involve considerable computational complexity and are limited to asymptotic convergence \cite{Zhu16tac}. Moreover, predictor feedback is also prone to numerical issues during implementation \cite{Furtat17tac,kara17book,Mondi04tac}.

Compared with traditional asymptotic stability or exponential stability,
finite-time control has garnered extensive attention in the control community
due to its faster convergence speed \cite{Bhat00siam}. Subsequently, a novel form of
finite-time control, called fixed-time control, was introduced in
\cite{Polyakov11tac}, where the upper bound of the convergence time is a
constant that is independent of the initial state.
Recently, a novel smooth fixed-time (FT) method utilizing artificial delay has been
proposed in \cite{Zhou21tac}, where the convergence time depends solely on the
designed artificial delay \cite{Kara06siam,Michiels20scl}. Further, a
Lyapunov-based approach for prescribed-time (in any prescribed time) stabilization by using
artificial delay was established in \cite{Ding23tac}. By using such artificial
delay, the prescribed-time control problem for nonholonomic integrator with
input delays was studied in \cite{Ding26auto}.

Since actual systems are always subject to external disturbances and internal
uncertainties, conducting an input-to-state stability (ISS) analysis is an
important task \cite{Stang95ijc}. Using the time-delay approach, the
conditions of ISS were investigated in \cite{Efimov23tac} for nonlinear
input-affine systems with respect to averaged value of the disturbances. A
new methodology to analyze ISS of nonlinear systems with time-varying delays
based on the use of nonlinear matrix inequalities was proposed in
\cite{emilia08auto}. Further, fractional power method based finite-time
input-to-state stability and fixed-time input-to-state stability were proposed
in \cite{hong10siam} and \cite{Efimov24siam}, respectively. Using the
trajectory-based approach proposed in \cite{Mazen14tac}, fixed-time
stabilization estimates using sampled controls for continuous-time
time-varying linear systems were given in \cite{Mazenc25tac}, and fixed-time
stabilization estimates for discrete-time linear systems were
given in \cite{Mazenc24css}.

In this paper, by using artificial delay, we study the fixed-time
input-to-state stabilization problem for perturbed linear systems with
nilpotent matrices under unknown constant delays in the full-state measurement or input. Here, we employ
an act-and-wait type controller structure, which was introduced in
\cite{Insp06tcst} and was used to achieve fixed-time control in
\cite{Zhou21tac,Michiels20scl} for linear systems with known input delay. Differently from \cite{Zhou21tac,Michiels20scl}, where the act time and wait time in the controller are the same, we adopt act and wait times of different
lengths to avoid the impact of measurement delays. When the linear system is known
exactly (except for the input/measurement delays), it is shown that the state can converge
to zero in fixed time, and the convergence time depends on the system and the upper bound of the unknown time delay. In the presence of additive disturbances and small system uncertainties, it can be proven that the closed-loop systems are almost fixed-time ISS in the
sense of \cite{Mazenc25tac}, whereas fixed-time ISS is recovered in the absence of system uncertainties. Further, the results are extended to
discrete-time systems. Compared with existing methods, the proposed controller offers two distinct advantages. On the one hand, in terms of convergence rate, compared with the low gain method \cite{zhou09tac}, which has a low convergence rate for arbitrarily large delays, and the PDE delay-adaptive control method \cite{Zhu16tac}, which is limited to asymptotic convergence, the controller proposed in this paper has a faster convergence rate and can achieve fixed-time stability. On the other hand, in terms of implementation, in contrast to traditional predictor methods \cite{Krstic10tac,Zhu16tac}, the controller proposed herein is finite-dimensional and concise in form, while accommodating unknown constant delays that can be arbitrarily large as long as their upper bounds are known.


\textbf{Notation} \textit{For any integers }$p$\textit{ and }$q$\textit{ with
}$p\leq q$\textit{, }$\mathbb{I}_{[p,q]}$ denotes the set $\{p,p+1,\ldots
,q\}$. $%
\mathbb{Z}
_{0}^{+}$ denotes the set of nonnegative integers\textit{. } \(\mathbb N\) denotes the set of positive integers\textit{. }$|\cdot|$\textit{
denotes the usual Euclidean norm on }$\mathbb{R}^{n}$\textit{ and the
corresponding induced matrix operator norm. }$|\cdot|_{S}$\textit{ denotes the
supremum over any interval }$S$\textit{ in this norm. }$\lfloor a\rfloor
$\textit{ }is the integer part of $a$.

\section{\label{sec2}Continuous-Time FT control}
To facilitate the subsequent controller design and stability analysis, we first introduce the weighting function $S(t)$ and recall the definition of fixed-time ISS.

\begin{definition}
\label{def1}\cite{Zhou21tac} A function $S:\mathbb{R}_{\geq0}\rightarrow\mathbb{R}$ is said to be an
admissible weighting function if it satisfies: 1). $S(t)$ is continuously differentiable and $3h$-periodic;
2). $S(t)\geq0$ for all $t\in[2h,3h]$, and there exists
$h_{\ast}\in(2h,3h)$ such that $S(h_{\ast})>0$;
3). $S(2h)=S(3h)=0$ and $S'(2h)=S'(3h)=0$.
\end{definition}
\begin{definition}\cite{Mazenc25tac}
\label{def2}Consider the time-delay system $\dot{x}(t)=f(t,x_{t},\delta(t))$,
where $x(t)\in\mathbb{R}^{n}$ and $\delta$ is piecewise continuous and
locally bounded. The system is said to be fixed-time input-to-state
stable (abbreviated as fixed-time ISS) if there exist a function $\varphi\in\mathcal{K}$ and constants
$T_{1}\geq T_{2}\geq0$ such that, for every admissible initial function
and every $\delta$, the corresponding solution is uniquely determined
and satisfies
$|x(t)|
\leq
\varphi(
|\delta|_{[t-T_{2},t]}
), t\geq T_{1}$.
\end{definition}
\subsection{Unknown Constant Delay: Nominal Case}

Consider the following linear continuous-time system
\begin{equation}
\left\{
\begin{array}
[c]{rl}%
\dot{x}(t) & =Ax(t)+B_{0}u(t),\\
y(t) & =x(t-\tau_{0}),
\end{array}
\right.  \label{sys_norm}%
\end{equation}
where $x(t)\in\mathbb{R}^{n}$, $A\in\mathbb{R}^{n\times n}$, $B_{0}%
\in\mathbb{R}^{n\times m}$, and $\tau_{0}\in\lbrack0,\tau_{M}]$ is an unknown
constant delay with known upper bound $\tau_{M}>0$. Here without loss of
generality, we set $y(s)=0$ for $s\leq0$.

\begin{assumption}
\label{ass0}$(A,B_{0})$ is controllable and $A$ is nilpotent with nilpotency
index $q\leq n$, namely, $A^{q}=0$ and $A^{q-1}\neq0$.
\end{assumption}
The controllability of $(A,B_{0})$ in Assumption \ref{ass0} can be verified by the
Kalman rank condition
\[
\begin{aligned}
&\operatorname{rank}
 \bigl[B_0,AB_0,\ldots,A^{n-1}B_0\bigr]\\
&\quad=\operatorname{rank}
 \bigl[B_0,AB_0,\ldots,A^{q-1}B_0\bigr]=n.
\end{aligned}
\]
Let $h\geq\tau_{M}$ be an artificial delay. Denote
\[
g\left(  t\right)  =\left\{
\begin{array}
[c]{cl}%
0, & t\in\lbrack3lh,(3l+2)h],\\
1, & t\in((3l+2)h,(3l+3)h),
\end{array}
\right.
\]
where $l\in%
\mathbb{Z}
_{0}^{+}$. Since the pair $(A,B_{0})$ is controllable, from Theorem 1 in
\cite{Zhou21tac} it follows that
\[
W_{\mathrm{c}}=\int_{2h}^{3h}\mathrm{e}^{-As}B_{0}S(s)B_{0}^{\mathrm{T}}\mathrm{e}^{-A^{\mathrm{T}}s}\mathrm{d}s,
\]
is invertible.

We present a controller leading to fixed-time stabilization in the following theorem:
\begin{theorem}
Consider system (\ref{sys_norm}) under Assumption \ref{ass0}. Given $h\geq
\tau_{M}$, consider the smooth time-varying controller
\begin{equation}
u(t)=-K\left(  t\right)  S(t)g\left(
t\right)  y\left(  t-h\right)  ,\quad t\geq0, \label{vv}%
\end{equation}
where $l=\lfloor t/(3h)\rfloor$ and
\begin{equation}
K\left(  t\right)  =B_{0}^{\mathrm{T}}\mathrm{e}^{-A^{\mathrm{T}}\left(
t-3lh\right)  }W_{\mathrm{c}}^{-1}\mathrm{e}^{-A\left(  t-3lh-h\right)  }.
\label{K_n}%
\end{equation}
Then the closed-loop system consisting of
(\ref{sys_norm}) and (\ref{vv}) satisfies
$x(t)=u(t)=0$ for all $t\geq3qh$ and is therefore fixed-time stable.
\end{theorem}
\begin{proof}
For $t\in\lbrack3lh,(3l+2)h]$, $l\in%
\mathbb{Z}
_{0}^{+}$, we have $g(t)=0$. Then system~(\ref{sys_norm}) under
controller~(\ref{vv}) can be represented as $\dot{x}(t)=Ax(t)$, whose solution is given
by
\begin{equation}
x(t)=\mathrm{e}^{A(t-3lh)}x(3lh). \label{11}%
\end{equation}
For $t\in((3l+2)h,(3l+3)h]$, $l\in%
\mathbb{Z}
_{0}^{+}$, we have $g(t)=1$, and the closed-loop system is expressed as
\begin{equation}
\dot{x}(t)=Ax(t)-B_{0}S(t)K\left(
t\right)  x\left(  t-h-\tau_{0}\right)  . \label{12}%
\end{equation}
By using the fact that $t-h-\tau_{0}\in(3lh,(3l+2)h]$, substituting (\ref{11})
into (\ref{12}), and employing the variation of constants formula, we have the
solution to (\ref{12}):
\begin{equation}
x\left(  t\right)  =\mathrm{e}^{A\left(  t-3lh\right)  }\mathit{\Delta
}(t)x\left(  3lh\right)  , \label{6}%
\end{equation}
where
\begin{align}
\mathit{\Delta}(t)=  I_{n}-\int_{2h}^{t-3lh}\mathrm{e}^{-As}B_{0}S(s) B_{0}^{\mathrm{T}}\mathrm{e}^{-A^{\mathrm{T}}s}\mathrm{d}%
sW_{\mathrm{c}}^{-1}\mathrm{e}^{-A\tau_{0}}, \label{delta_nom}%
\end{align}
which implies that
$x((3l+3)h)  =\mathrm{e}^{3Ah}(  I_{n}-\mathrm{e}^{-A\tau_{0}})  x(
3lh)$.
By iterating this, we obtain
\begin{equation}
x\left(  3hl\right)  =\left(  \mathrm{e}^{3Ah}\left(  I_{n}-\mathrm{e}%
^{-A\tau_{0}}\right)  \right)  ^{l}x\left(  0\right)  .\label{x_3lh00}%
\end{equation}
Since $A^{q}=0$, we have
$I_{n}-\mathrm{e}^{-A\tau_{0}}=A\tau_{0}-(A\tau_{0})^{2}/2!+\cdots-((-A\tau_{0})  ^{q-1})/(q-1)!,
$
which implies that $(I_{n}-\mathrm{e}^{-A\tau_{0}}) ^{q}=0$, and thus $x(3hq)=0$. From (\ref{11}) and (\ref{6}) with $l=q$, it follows $x(t)=\mathrm{e}^{A(t-3qh)}x(3qh)$ for $t\in(3qh,(3q+2)h]$ and $x(t)=\mathrm{e}^{A(t-3qh)}\mathit{\Delta}(t)x(3qh)=0$ for $t\in((3q+2)h,(3q+3)h]$ and $u(t)=-K(t)S(t)y(t-h)=-K(t)S(t)x(t-\tau_{0}-h)=0$.
Repeating the above process, for all $t\geq3qh$ we can also obtain that
$x(t)=0$ and $u(t)=0$.
\end{proof}
\begin{remark}
Assumption~1 is the same as in~\cite{zhou09tac}. System~(\ref{sys_norm}) under Assumption~\ref{ass0} can be transformed into one or more integrator chains via an invertible state transformation. The reason we adopt the form of system~(1) in this paper is to provide a more direct approach that requires no transformation. Moreover, later sections demonstrate that under perturbations of system~(1), the designed controller achieves exponential convergence, while the controller in~\cite{zhou09tac} only attains asymptotic convergence under exact matching of $A$ and $B$.
\end{remark}
\begin{remark}
Our main result concerns unknown delay compensation with fast convergence, which is challenging. Regarding the advantages of
fixed-time over finite-time convergence, on a compact initial set
$\mathcal{X}_0$, a finite-time controller may also admit a uniform
settling-time bound and yield a smaller estimate near the origin. The
advantage here is the explicit a priori bound $3qh$, which is
independent of the actual initial condition and $\mathcal{X}_0$ and
requires no estimation of the settling-time function. This bound
remains valid if $\mathcal{X}_0$ is enlarged, provided that the
required control amplitude is available.
\end{remark}

\subsection{Unknown Constant Delay: Perturbed Case}

Consider the following perturbed continuous-time system
\begin{equation}
\left\{
\begin{array}
[c]{rl}%
\dot{x}(t) & =Ax(t)+Bu(t)+f\left(  t,x(t)\right)  +w_{1}\left(  t\right)  ,\\
y(t) & =x(t-\tau_{0})+w_{2}\left(  t\right)  ,
\end{array}
\right.  \label{sss1}%
\end{equation}
where $w_{1}(t)$ and $w_{2}(t)$ are piecewise continuous and locally bounded disturbances. Function
$f(t,x)$ is unknown, continuous in its first argument and locally Lipschitz in
its second argument and is subject to sector bound $|f(t,x(t))|\leq
f_{0}|x(t)|$ where $f_{0}\geq0$ is known. $B$ has the form $B=B_{0}+\mathit{\Delta
}B$, where $B_{0}$ is constant, known and such that $|\mathit{\Delta}%
B|\leq\varepsilon$ where $\varepsilon\geq0$ is known.

Denote
\[
\begin{array}{ll}
\rho_{0}(l)  &=  \eta\left\vert F^{l}\right\vert ,\\
\rho_{1}  &=  \alpha_{1}\eta\sup_{l\in%
\mathbb{N}
}\left\{  \sum_{i=0}^{l-1}\,\left\vert F^{i}\right\vert \mathrm{e}%
^{3\left\vert A\right\vert h}\right\} \\
&  +\mathrm{e}^{3\left\vert A\right\vert h}\int_{0}^{3h}\mathrm{e}^{\left\vert
A\right\vert s}\mathrm{d}s+\mathrm{e}^{3\left\vert A\right\vert h}\alpha_{1},
\end{array}
\]
in which $s=t-3lh$, and
$\mathit{\Delta}(3lh+s)$ and $\mathit{\Theta}(3lh+s)$ below are independent
of $l$,
\begin{align}
&\eta=
\begin{array}{ll}
	\sup\limits_{s\in\lbrack2h,3h]}
	\left\{\mathrm{e}^{3|A|h}
	+\mathrm{e}^{3|A|h}|\mathit{\Delta}(3lh+s)|
	+|\mathit{\Theta}(3lh+s)|\right\}
\end{array},\nonumber\\
&\alpha_{1}=
\begin{array}{ll}
	\int_{2h}^{3h}|\mathrm{e}^{-As}BS(s)
	B_{0}^{\mathrm{T}}\mathrm{e}^{-A^{\mathrm{T}}s}
	W_{\mathrm{c}}^{-1}\mathrm{e}^{-A(s-h)}|
\end{array}\nonumber\\
&\qquad\begin{array}{ll}
	\times\int_{0}^{s-h-\tau_{0}}
	|\mathrm{e}^{A(s-h-\tau_{0}-r)}|\mathrm{d}r\mathrm{d}s
	+\int_{0}^{3h}\mathrm{e}^{|A|s}\mathrm{d}s,
\end{array}\label{alpha_01}\\
&\mathit{F}=
\begin{array}{ll}
	\mathrm{e}^{3Ah}(I_{n}-\mathrm{e}^{-A\tau_{0}})
	-\mathrm{e}^{3Ah}\int_{2h}^{3h}
	\mathrm{e}^{-As}\mathit{\Delta}B S(s)B_{0}^{\mathrm{T}}
\end{array}\nonumber\\
&\qquad\begin{array}{ll}
	\times \mathrm{e}^{-A^{\mathrm{T}}s}\mathrm{d}sW_{\mathrm{c}}^{-1}
	\mathrm{e}^{-A\tau_{0}}
\end{array},\label{F_def}\\
&\mathit{\Delta}(3lh+s)=
\begin{array}{ll}
	I_{n}-\int_{2h}^{s}\mathrm{e}^{-Ar}B_{0}
	S(r)B_{0}^{\mathrm{T}}\mathrm{e}^{-A^{\mathrm{T}}r}
	\mathrm{d}r
\end{array}\nonumber\\
&\qquad\begin{array}{ll}
	\times \,W_{\mathrm{c}}^{-1}\mathrm{e}^{-A\tau_{0}},
\end{array}\label{delta}\\
&\mathit{\Theta}(3lh+s)=
\begin{array}{ll}
	-\int_{2h}^{s}\mathrm{e}^{A(s-r)}
	\mathit{\Delta}B S(r)B_{0}^{\mathrm{T}}\mathrm{e}^{-A^{\mathrm{T}}r}
	\mathrm{d}r
\end{array}\nonumber\\
&\qquad\begin{array}{ll}
	\times \,W_{\mathrm{c}}^{-1}\mathrm{e}^{-A\tau_{0}}
\end{array}.\label{ceta_c}
\end{align}

\begin{assumption}
\label{ass2}
Assumption~\ref{ass0} is satisfied and
$\rho(l)=\rho_{0}(l)+\rho_{1}f_{0}<1$
holds for all $\tau_{0}\in[0,\tau_{M}]$,
all $\mathit{\Delta}B$ satisfying
$|\mathit{\Delta}B|\leq\varepsilon$,
and all $l\geq q$.
\end{assumption}

\begin{remark}
\label{remark_1}
The above assumptions are similar to Assumption~3 in
\cite{Mazenc25tac}. We next provide conservative but computable sufficient
bounds on $\varepsilon$ and $f_{0}$ that guarantee
$\rho(l)<1$ for all $l\geq q$.

Let
$F_{0}(\tau_{0})=
\mathrm{e}^{3Ah}
\left(
I_{n}-\mathrm{e}^{-A\tau_{0}}
\right)$.
Since $A^{q}=0$ and the two factors in $F_{0}(\tau_{0})$ commute,
$F_{0}^{q}(\tau_{0})=0$ for every
$\tau_{0}\in[0,\tau_{M}]$. Moreover, from (\ref{ceta_c}),
$
\left|
\mathit{\Theta}(3lh+s)
\right|
\leq
c_{\mathit{\Theta}}
\left|
\mathit{\Delta}B
\right|$, $
s\in[2h,3h]$,
where
\[
c_{\mathit{\Theta}}
=
\sup_{\substack{
\theta\in[2h,3h]\\
\tau_{0}\in[0,\tau_{M}]
}}
\int_{2h}^{\theta}
\left|
\mathrm{e}^{A(\theta-\sigma)}
\right|
S(\sigma)
|K(\sigma)|
\left|
\mathrm{e}^{A(\sigma-h-\tau_{0})}
\right|
\,\mathrm{d}\sigma.
\]
In particular, since $F
=
F_{0}(\tau_{0})
+
\mathit{\Theta}((3l+3)h)$,
the condition
$|\mathit{\Delta}B|\leq\varepsilon$ implies
$
\left|
F-F_{0}(\tau_{0})
\right|
\leq
c_{\mathit{\Theta}}\varepsilon$.

Define
\[
M
=
1+
\sup_{\tau_{0}\in[0,\tau_{M}]}
\left|
F_{0}(\tau_{0})
\right|.
\]
If
$\varepsilon\leq1/c_{\mathit{\Theta}}$,
then
$|F|\leq M,
|F_{0}(\tau_{0})|\leq M$.
For brevity, write $F_{0}=F_{0}(\tau_{0})$. Using
\[
F^{q}-F_{0}^{q}
=
\sum_{j=0}^{q-1}
F^{q-1-j}
(F-F_{0})
F_{0}^{j},
\]
we obtain $
|F^{q}|
\leq
qM^{q-1}
c_{\mathit{\Theta}}
\varepsilon$.
Therefore,
\[
\varepsilon
<
\frac{1}{
2qM^{q-1}c_{\mathit{\Theta}}
}
\]
implies $|F^{q}|<1/2$.

For every $l\geq q$, write
$l=kq+j,
k\geq1,
j\in\{0,\ldots,q-1\}$.
Since $M\geq1$, it follows that
\[
|F^{l}|
\leq
|F^{q}|^{k}|F^{j}|
\leq
qM^{2q-2}
c_{\mathit{\Theta}}
\varepsilon,
\]
and
\[
\sum_{i=0}^{\infty}|F^{i}|
\leq
\sum_{k=0}^{\infty}|F^{q}|^{k}
\sum_{j=0}^{q-1}M^{j}
\leq
2\sum_{j=0}^{q-1}M^{j}.
\]

Let $\bar{\eta}$ and $\bar{\alpha}_{1}$ be computable uniform
upper bounds on $\eta$ and $\alpha_{1}$, respectively, over
$\tau_{0}\in[0,\tau_{M}]$, $|\mathit{\Delta}B|
\leq
1/c_{\mathit{\Theta}}.
$
Then, for every $l\geq q$,
\[
\rho_{0}(l)
=
\eta|F^{l}|
\leq
\bar{\eta}
qM^{2q-2}
c_{\mathit{\Theta}}
\varepsilon,
\]
and
\[
\rho_{1}
\leq
\mathrm{e}^{3|A|h}
\left(
2\bar{\alpha}_{1}\bar{\eta}
\sum_{j=0}^{q-1}M^{j}
+
\int_{0}^{3h}
\mathrm{e}^{|A|s}\,\mathrm{d}s
+
\bar{\alpha}_{1}
\right).
\]

Consequently, it is sufficient to choose
\[
0\leq\varepsilon
<
\min
\left\{
\frac{1}{
2qM^{q-1}c_{\mathit{\Theta}}
},
\frac{1}{
\bar{\eta}qM^{2q-2}c_{\mathit{\Theta}}
}
\right\}
\]
and
\[
0\leq f_{0}
<
\frac{
1-
\bar{\eta}qM^{2q-2}
c_{\mathit{\Theta}}
\varepsilon
}{
\mathrm{e}^{3|A|h}
\left(
2\bar{\alpha}_{1}\bar{\eta}
\displaystyle\sum_{j=0}^{q-1}M^{j}
+
\displaystyle\int_{0}^{3h}
\mathrm{e}^{|A|s}\,\mathrm{d}s
+
\bar{\alpha}_{1}
\right)
}.
\]
The first bound on $\varepsilon$ also guarantees
$\varepsilon<1/c_{\mathit{\Theta}}$. Hence, these bounds ensure that
\[
\rho(l)
=
\rho_{0}(l)+\rho_{1}f_{0}
<1,
\qquad
l\geq q.
\]
All the quantities in these bounds are computable from the known
system matrices, the delay bound $\tau_{M}$, the artificial delay
$h$, and the selected weighting function $S$.
\end{remark}

\begin{remark}
Theorem~\ref{the_4} also applies to
$\widetilde A=A+\mathit{\Delta}A$, where $A$ is nilpotent and
$(A,B_0)$ is controllable. Designing the controller with $A$, we absorb
$\mathit{\Delta}Ax$ into
$\widetilde f(t,x)=f(t,x)+\mathit{\Delta}Ax$. Since
$|\widetilde f(t,x)|\leq(f_0+|\mathit{\Delta}A|)|x|$, the result holds
if Assumption~\ref{ass2} is satisfied with $f_0$ replaced by
$f_0+|\mathit{\Delta}A|$. Thus, $\widetilde A$ need not be nilpotent,
and the resulting property is almost fixed-time ISS rather than exact
fixed-time stabilization.
\end{remark}

Denote%
\begin{align}\label{delta_c}
\begin{array}{ll}
	\delta=  &  \eta\sup_{l\in%
\mathbb{Z}
_{0}^{+}}\left\{  \sum_{i=0}^{l-1}\,\left\vert F^{i}\right\vert \right\}
\mathrm{e}^{3\left\vert A\right\vert h}\left(  \alpha_{1}+\alpha_{2}\right)\\
&  +\mathrm{e}^{3\left\vert A\right\vert h}\int_{0}^{3h}\mathrm{e}^{\left\vert
A\right\vert s}\mathrm{d}s+\mathrm{e}^{3\left\vert A\right\vert h}\left(
\alpha_{1}+\alpha_{2}\right)  , %
\end{array}
\end{align}
where%
\begin{align}\label{alpha_02}
\begin{array}{ll}
\alpha_{2}=\int_{2h}^{3h}\left\vert \mathrm{e}^{-As}BS(s)B_{0}^{\mathrm{T}}\mathrm{e}^{-A^{\mathrm{T}}s}W_{\mathrm{c}%
}^{-1}\mathrm{e}^{-A(s-h)}\right\vert \mathrm{d}s. %
\end{array}
\end{align}
Following \cite{Mazenc25tac}, an ISS estimate is called an almost fixed-time
ISS estimate if the time after which the estimate applies is uniform with
respect to the initial condition and its guaranteed exponential convergence
rate tends to infinity as the uncertainty bounds tend to zero.

We are in a position to state the following result:

\begin{theorem}
\label{the_4}Consider system (\ref{sss1}) under Assumption \ref{ass2}. Given
$h\geq\tau_{M}$, consider the smooth time-varying controller (\ref{vv}). Then
the solution of system~(\ref{sss1}) under controller~(\ref{vv}) satisfies
\begin{equation}
|x(t)|\leq|x|_{[0,3(q+1)h]}\mathrm{e}^{\frac{\ln(\rho(q))(t-3(q+1)h)}{3(q+1)h}}%
+\frac{\delta(|w_{1}|_{[0,t]}+|w_{2}|_{[0,t]})}{(1-\rho(q))^{2}}, \label{ni_A_1}%
\end{equation}
for all $t\geq3(q+1)h$. Moreover, as $\varepsilon$ and $f_{0}$ tend to zero, the derived upper bound
on $\rho(q)$ tends to zero, and hence the exponential convergence rate in
\eqref{ni_A_1} tends to infinity. Therefore, \eqref{ni_A_1} is an almost
fixed-time ISS estimate in the sense defined above. In particular, if $\mathit{\Delta}B=0$ and $f_{0}=0$, we
have $\rho(q)=0$ and thus $|x(t)|\leq\delta(|w_{1}|_{[t-3(q+1)h,t]}+|w_{2}|_{[t-3(q+1)h,t]})$ for
$t\geq3(q+1)h$, which establishes fixed-time ISS in the
sense of Definition~\ref{def2}.
\end{theorem}

For readability, we place the proof in Appendix A1.

\begin{remark}
When \(\tau_{0}=0\), the definition of \(\mathit{\Delta}\) gives
$\mathit{\Delta}((3l+3)h)=I_{n}-\mathrm{e}^{-A\tau_{0}}=0$.
Hence, the cycle-boundary recursion in Appendix~A1 yields
$
x(3hl)=\mathit{\Theta}_{\ast}^{\,l}x(0)
+\sum_{i=1}^{l}\mathit{\Theta}_{\ast}^{\,l-i}
\mathrm{e}^{3Ah}\mathit{\Pi}_{w}(3hi),
$
where
\(\mathit{\Theta}_{\ast}=\mathit{\Theta}((3l+3)h)\) is independent
of \(l\). Therefore, if \(f(t,x)=0\) and
\(\lvert\mathit{\Theta}_{\ast}\rvert\leq\bar{\theta}<1\), then
$|x(3hl)|
\leq{}\bar{\theta}^{\,l}|x(0)|
+\frac{1-\bar{\theta}^{\,l}}{1-\bar{\theta}}
\mathrm{e}^{3|A|h}(
\alpha_{1}|w_{1}|_{[0,3hl]}
+\alpha_{2}|w_{2}|_{[0,3hl]}
)$.
This is the cycle-boundary counterpart of \cite[(23)]{Mazenc25tac}.
\end{remark}

\begin{remark}
(Unknown constant input delay) If the input of system (\ref{sss1}) is subject to an unknown constant delay $\tau_{1}\leq\tau_{M1}$
with $\tau_{M1}$ being known, i.e.\ $u(t-\tau_{1})$, it can be shifted to the
measurement side by a time-axis transformation, i.e.\ $z(t)=x(t+\tau_{1})$.
Then $y(t)=z(t-\tau_{0}-\tau_{1})$, and hence the total unknown
measurement delay for the transformed system is bounded by
$\tau_{M}+\tau_{M1}$. By applying Theorem \ref{the_4}, we obtain the
following control input:
\[
u(t)=-K\left(  t\right) S(t)g\left(
t\right)  y\left(  t-h\right)  ,\quad t\geq0,
\]
with $h\geq\tau_{M1}+\tau_{M}$. In addition, $3(q+1)h$ in (\ref{ni_A_1}) is
replaced by $3(q+1)h+\tau_{M1}$.
\end{remark}
\subsection{Single Integrator with Unknown Fast-Varying Delay}

If $A=0$, system (\ref{sss1}) reduces to a single integrator. Different from
the previous section, here we consider the case of an unknown fast-varying
delay (piecewise-continuous delay without any assumptions on the delay
derivative), i.e.,
\begin{equation}
\left\{
\begin{array}
[c]{rl}%
\dot{x}(t) & =Bu(t)+f\left(  t,x\right)  +w_{1}\left(  t\right)  ,\\
y(t) & =x(t-\tau(t))+w_{2}\left(  t\right)  ,
\end{array}
\right.  \label{scalar}%
\end{equation}
where $x(t)\in\mathbb{R}$, $B\in\mathbb{R},\tau(t)\leq$\ $\tau_{M}$ is an
unknown fast-varying delay.

We are in a position to state the following result:
\begin{proposition}
\label{corl1}Consider system (\ref{scalar}) under Assumption \ref{ass2} with
$A=0$. Given $h\geq\tau_{M}$. Consider the controller (\ref{vv}) with $A=0$.
Then the solution of system~(\ref{scalar}) under controller~(\ref{vv})
satisfies (\ref{ni_A_1}).
\end{proposition}
\begin{proof}
Following the argument of Theorem~\ref{the_4}, the matrix
$\mathit{\Delta}((3l+3)h)$ defined in (\ref{delta}) for a constant delay
is modified as follows for the time-varying delay $\tau(t)$:
\begin{align*}
\mathit{\Delta}((3l+3)h)=  &  I_{n}-\int_{(3l+2)h}^{(3l+3)h}\mathrm{e}%
^{-A(s-3lh)}B_{0}S(s)\,\\
&  \times B_{0}^{\mathrm{T}}\mathrm{e}^{-A^{\mathrm{T}}(s-3lh)}W_{\mathrm{c}%
}^{-1}\mathrm{e}^{-A\tau(s-h)}\mathrm{d}s.
\end{align*}
If $A=0$, $\mathrm{e}^{-A\tau(s-h)}=1$. Hence the expression simplifies to
$\mathit{\Delta}((3l+3)h)=0$ which is independent of the delay $\tau(s)$.
Consequently, the time-varying nature of $\tau(s)$ has no effect on
$\mathit{\Delta}((3l+3)h)$ when $A=0$. The remainder of the proof follows that of Theorem \ref{the_4} and is omitted here due to space constraints.
\end{proof}
\begin{remark}
\label{remark_tv_delay}
The restriction $A=0$ in Proposition~\ref{corl1} is required only when the
fast-varying measurement delay is unknown. For a general $A\neq0$, a known
fast-varying measurement delay can be handled by replacing $K(t)$ in the
controller with
\[
K_{\tau}(t)=K(t)\mathrm{e}^{A\tau(t-h)}.
\]
Indeed, for
$t\in[(3l+2)h,(3l+3)h]$, we have
$t-h-\tau(t-h)\in[3lh,3lh+2h]$,
which belongs to the preceding wait interval. Therefore, for the nominal
part of the measurement,
\[
\begin{aligned}
y(t-h)
&=x\bigl(t-h-\tau(t-h)\bigr)\\
&=\mathrm{e}^{A(t-h-\tau(t-h)-3lh)}x(3lh)\\
&=\mathrm{e}^{-A\tau(t-h)}x(t-h).
\end{aligned}
\]
Thus, the additional factor in $K_{\tau}(t)$ cancels the
delay-dependent factor generated by the nominal delayed measurement. Consequently, the terminal matrix reduces to
\[
\mathit{\Delta}((3l+3)h)
=I_n-W_{\mathrm{c}}W_{\mathrm{c}}^{-1}=0.
\]
Thus, no bound on the variation rate of the known delay is needed. In
contrast, if $\tau(t)$ is unknown and $A\neq0$, the factor
$\mathrm{e}^{-A\tau(s-h)}$ remains inside the integral defining
$\mathit{\Delta}((3l+3)h)$. Unlike the unknown constant-delay case, this
factor cannot be taken outside the integral, and the resulting terminal
matrix depends on the entire unknown delay profile and is not generally
nilpotent. Therefore, the present approach does not guarantee the estimate
of Theorem~\ref{the_4} for a general $A\neq0$ under an unknown fast-varying
measurement delay. When $A=0$, this factor is identically equal to one,
which explains why Proposition~\ref{corl1} allows an unknown fast-varying
measurement delay. Finally, an unknown fast-varying input delay would
replace the prescribed switching signal $g(t)$ at the plant input by
$g(t-\tau(t))$, thereby destroying the fixed act-and-wait partition
used in the proof; such input delays are therefore not covered by the
present approach.
\end{remark}
\section{\label{sec3}Discrete-Time FT control}
We first recall the definition of fixed-time ISS for discrete-time systems.
\begin{definition}
\label{def_discrete_ftiss}
\cite{Mazenc24css} Consider the discrete-time delay system
$x(k+1)=\mathcal{F}_{\mathrm d}
\bigl(k,x_{k,d},v_{\mathrm d}(k)\bigr)$,
where $x(k)\in\mathbb{R}^{n}$, $d\in\mathbb{Z}_{0}^{+}$ is a constant
integer delay,
$x_{k,d}(m)=x(k+m)$ for $m\in\mathbb{I}_{[-d,0]}$, and
$v_{\mathrm d}$ is a disturbance sequence. The system is said to be
fixed-time input-to-state stable (fixed-time ISS) if there exist a function
$\varphi\in\mathcal{K}$ and nonnegative integers $T_{1}\geq T_{2}$ such
that, for every admissible initial sequence and every $v_{\mathrm d}$, the
corresponding solution is uniquely determined and satisfies
\[
|x(k)|
\leq
\varphi\left(
|v_{\mathrm d}|_{[k-T_{2},k]}
\right),
\qquad k\geq T_{1}.
\]
\end{definition}
\subsection{Full State-Measurements with Unknown Constant Delay}

Consider the following linear discrete-time system
\begin{equation}
\left\{
\begin{array}
[c]{rl}%
x(k+1) & =A_{\mathrm{d}}x(k)+B_{\mathrm{d}}u(k)+f_{\mathrm{d}}%
(k,x(k))+w_{\mathrm{d}1}(k),\\
y(k) & =x(k-d_{0})+w_{\mathrm{d}2}(k),
\end{array}
\right.  \label{dis_sys_2}%
\end{equation}
where $x(k)\in\mathbb{R}^{n}$, $A_{\mathrm{d}}\in\mathbb{R}^{n\times n}$,
$B_{\mathrm{d}}\in\mathbb{R}^{n\times m}$, $d_{0}\in\mathbb{I}_{[0,D_{M}]}$ is
an unknown integer with $D_{M}$ being a known nonnegative integer,\ $w_{\mathrm{d}1}(k)$
and $w_{\mathrm{d}2}(k)$ are unknown locally bounded disturbance sequences, and $f_{\mathrm{d}}(k,x(k))$ is an unknown function and
$|f_{\mathrm{d}}(k,x(k))|\leq f_{\mathrm{d}0}|x(k)|$ where $f_{\mathrm{d}0}\geq0$
is known. $B_{\mathrm{d}}$ has the form $B_{\mathrm{d}}=B_{\mathrm{d}%
0}+\mathit{\Delta}B_{\mathrm{d}}$, where $B_{\mathrm{d}0}$ is constant, known
and such that $|\mathit{\Delta}B_{\mathrm{d}}|\leq\varepsilon_{\mathrm{d}}$
with $\varepsilon_{\mathrm{d}}\geq0$ known.

Let $h\geq\max\{D_{M},n\}$ be an artificial positive integer delay. Denote
\[
p\left(  k\right)  =\left\{
\begin{array}
[c]{cl}%
0, & k\in\mathbb{I}_{[3lh,3lh+2h-1]},\\
1, & k\in\mathbb{I}_{[3lh+2h,3lh+3h-1]},
\end{array}
\right.
\]
where $l=\lfloor k/(3h)\rfloor$.

We make the following assumption.

\begin{assumption}
\label{assu_3}$A_{\mathrm{d}}-I_{n}$ is nilpotent with nilpotency index $q_{\mathrm{d}}\leq n$, i.e., $(A_{\mathrm{d}}%
-I_{n})^{q_{\mathrm{d}}}=0$ and $(A_{\mathrm{d}}-I_{n})^{q_{\mathrm{d}}-1}\neq0$, and $(A_{\mathrm{d}},B_{\mathrm{d}0})$ is controllable.
\end{assumption}

Since the pair $(A_{\mathrm{d}},B_{\mathrm{d}0})$ is controllable and $h\geq
n$, we can get that%
\[
W_{\mathrm{d}}=%
{\displaystyle\sum_{k=0}^{h-1}}
A_{\mathrm{d}}^{k}B_{\mathrm{d}0}B_{\mathrm{d}0}^{\mathrm{T}}\left(
A_{\mathrm{d}}^{\mathrm{T}}\right)  ^{k},
\]
is invertible.

Since $A_{\mathrm{d}}-I_{n}$ is nilpotent, $A_{\mathrm{d}}$ is
invertible, and hence its negative powers used below are well-defined.
Denote
\[
\begin{array}{ll}
\rho_{\mathrm{d}0}(l)
&=\eta_{\mathrm{d}}|F_{\mathrm{d}}^{l}|,\\
\rho_{\mathrm{d}1}
&=\eta_{\mathrm{d}}\alpha_{\mathrm{d}1}
\sup\limits_{l\in\mathbb{N}}
\left\{\textstyle\sum_{i=0}^{l-1}|F_{\mathrm{d}}^{\,i}|\right\}
+\textstyle\sum_{i=0}^{3h-1}|A_{\mathrm{d}}^{i}|
+\alpha_{\mathrm{d}1},
\end{array}
\]
in which
\begin{align}
&\eta_{\mathrm d}=
\begin{array}{ll}
\scalebox{0.9}{$\displaystyle
\max\left\{
\begin{array}{l}
\displaystyle
\sup_{j\in\mathbb{I}_{[0,2h-1]}}
|A_{\mathrm d}^{j+1}|,
\end{array}
\right.
$}
\end{array}\nonumber\\
&\qquad
\begin{array}{ll}
\scalebox{0.9}{$\displaystyle
\left.
\begin{array}{l}
\displaystyle
\sup_{j\in\mathbb{I}_{[2h,3h-1]}}
\left(
|A_{\mathrm d}^{j+1}
\mathit{\Delta}_{\mathrm d}(3hl+j)|
+
|\mathit{\Theta}_{\mathrm d}(3hl+j)|
\right)
\end{array}
\right\}
$}
\end{array},\nonumber\\
&\alpha_{\mathrm{d}1}=
\begin{array}{ll}
\scalebox{0.9}{$\displaystyle
\sup_{k\in\mathbb{I}_{[2h,3h-1]}}
\Bigg\{\sum_{i=0}^{k}|A_{\mathrm{d}}^{i}|
+\sum_{i=2h}^{k}
|A_{\mathrm{d}}^{k-i}B_{\mathrm{d}} B_{\mathrm{d}0}^{\mathrm{T}}
$}
\end{array}\nonumber\\
&\qquad
\begin{array}{ll}
\scalebox{0.85}{$\displaystyle
\times
(A_{\mathrm{d}}^{3h-i-1})^{\mathrm{T}}
W_{\mathrm{d}}^{-1}A_{\mathrm{d}}^{5h-i}|
\sum_{r=0}^{i-d_{0}-h-1}
|A_{\mathrm{d}}^{i-d_{0}-h-1-r}|\Bigg\}
$}
\end{array},
\label{alpha_1}\\
&F_{\mathrm{d}}=
\begin{array}{ll}
A_{\mathrm{d}}^{3h}(I_{n}-A_{\mathrm{d}}^{h-d_{0}})
-\textstyle\sum_{r=2h}^{3h-1}
A_{\mathrm{d}}^{3h-r-1}\mathit{\Delta}B_{\mathrm{d}}
B_{\mathrm{d}0}^{\mathrm{T}}
\end{array}\nonumber\\
&\qquad
\begin{array}{ll}
\times
(A_{\mathrm{d}}^{3h-r-1})^{\mathrm{T}}
W_{\mathrm{d}}^{-1}
A_{\mathrm{d}}^{4h-d_{0}},
\end{array}
\label{F_d_def}\\
&\mathit{\Delta}_{\mathrm{d}}(3hl+j)=
\begin{array}{ll}
I_{n}
-
\textstyle\sum_{r=2h}^{j}
A_{\mathrm{d}}^{-r-1}B_{\mathrm{d}0}
B_{\mathrm{d}0}^{\mathrm{T}}
(A_{\mathrm{d}}^{3h-r-1})^{\mathrm{T}}
\end{array}\nonumber\\
&\qquad
\begin{array}{ll}
\times
W_{\mathrm{d}}^{-1}
A_{\mathrm{d}}^{4h-d_{0}},
\end{array}
\label{delta_ds}\\
&\mathit{\Theta}_{\mathrm{d}}(3hl+j)=
\begin{array}{ll}
-\textstyle\sum_{r=2h}^{j}
A_{\mathrm{d}}^{j-r}\mathit{\Delta}B_{\mathrm{d}}
B_{\mathrm{d}0}^{\mathrm{T}}
(A_{\mathrm{d}}^{3h-r-1})^{\mathrm{T}}
\end{array}\nonumber\\
&\qquad
\begin{array}{ll}
\times
W_{\mathrm{d}}^{-1}
A_{\mathrm{d}}^{4h-d_{0}}.
\end{array}
\label{ceta_ds}
\end{align}
A sum is understood to be zero whenever its upper limit is smaller than
its lower limit.
\begin{assumption}
\label{ass4}
Assumption~\ref{assu_3} is satisfied and
$\rho_{\mathrm{d}}(l)
=\rho_{\mathrm{d}0}(l)+\rho_{\mathrm{d}1}f_{\mathrm{d}0}<1$
holds for all $d_{0}\in\mathbb{I}_{[0,D_{M}]}$,
all $\mathit{\Delta}B_{\mathrm{d}}$ satisfying
$|\mathit{\Delta}B_{\mathrm{d}}|\leq\varepsilon_{\mathrm{d}}$,
and all $l\geq q_{\mathrm{d}}$.
\end{assumption}

\begin{remark}
Similar to Remark~\ref{remark_1}, we provide conservative but
computable sufficient bounds on $\varepsilon_{\mathrm d}$ and
$f_{\mathrm d0}$ that guarantee $\rho_{\mathrm d}(l)<1$ for all
$l\geq q_{\mathrm d}$.

For each $d_{0}\in\mathbb{I}_{[0,D_{M}]}$, define the nominal
period map
\[
F_{\mathrm d0}(d_{0})
=
A_{\mathrm d}^{3h}
\left(
I_{n}-A_{\mathrm d}^{h-d_{0}}
\right).
\]
Since
$(A_{\mathrm d}-I_{n})^{q_{\mathrm d}}=0$, the matrix
$I_{n}-A_{\mathrm d}^{h-d_{0}}$ is nilpotent with index no
greater than $q_{\mathrm d}$. Moreover, the two factors in
$F_{\mathrm d0}(d_{0})$ commute, and hence
$F_{\mathrm d0}^{q_{\mathrm d}}(d_{0})=0$, $d_{0}\in\mathbb{I}_{[0,D_{M}]}$.

From (\ref{ceta_ds}),
\[
\left|
\mathit{\Theta}_{\mathrm d}(3hl+j)
\right|
\leq
c_{\mathit{\Theta}_{\mathrm d}}
\left|
\mathit{\Delta}B_{\mathrm d}
\right|,
\qquad
j\in\mathbb{I}_{[2h,3h-1]},
\]
where
\[
\begin{aligned}
c_{\mathit{\Theta}_{\mathrm d}}
=
\max_{\substack{
d_{0}\in\mathbb{I}_{[0,D_{M}]}\\
j\in\mathbb{I}_{[2h,3h-1]}
}}
\sum_{r=2h}^{j}
&\left|
A_{\mathrm d}^{j-r}
\right|
\left|
B_{\mathrm d0}^{\mathrm T}
(A_{\mathrm d}^{3h-r-1})^{\mathrm T}
W_{\mathrm d}^{-1}
A_{\mathrm d}^{4h-d_{0}}
\right|.
\end{aligned}
\]
In particular,
\[
F_{\mathrm d}
=
F_{\mathrm d0}(d_{0})
+
\mathit{\Theta}_{\mathrm d}((3l+3)h-1),
\]
and therefore
$|\mathit{\Delta}B_{\mathrm d}|
\leq\varepsilon_{\mathrm d}$ implies
$\left|
F_{\mathrm d}-F_{\mathrm d0}(d_{0})
\right|
\leq
c_{\mathit{\Theta}_{\mathrm d}}
\varepsilon_{\mathrm d}$.

Define
\[
M_{\mathrm d}
=
1+
\max_{d_{0}\in\mathbb{I}_{[0,D_{M}]}}
\left|
F_{\mathrm d0}(d_{0})
\right|.
\]
If
$\varepsilon_{\mathrm d}
\leq
1/c_{\mathit{\Theta}_{\mathrm d}}$,
then
\[
|F_{\mathrm d}|\leq M_{\mathrm d},
\qquad
|F_{\mathrm d0}(d_{0})|\leq M_{\mathrm d}.
\]
For brevity, write
$F_{\mathrm d0}=F_{\mathrm d0}(d_{0})$. Using
\[
F_{\mathrm d}^{q_{\mathrm d}}
-
F_{\mathrm d0}^{q_{\mathrm d}}
=
\sum_{j=0}^{q_{\mathrm d}-1}
F_{\mathrm d}^{q_{\mathrm d}-1-j}
(F_{\mathrm d}-F_{\mathrm d0})
F_{\mathrm d0}^{j},
\]
we obtain
\[
|F_{\mathrm d}^{q_{\mathrm d}}|
\leq
q_{\mathrm d}
M_{\mathrm d}^{q_{\mathrm d}-1}
c_{\mathit{\Theta}_{\mathrm d}}
\varepsilon_{\mathrm d}.
\]
Consequently,
\[
\varepsilon_{\mathrm d}
<
\frac{1}{
2q_{\mathrm d}
M_{\mathrm d}^{q_{\mathrm d}-1}
c_{\mathit{\Theta}_{\mathrm d}}
}
\]
implies
$|F_{\mathrm d}^{q_{\mathrm d}}|<1/2$.

For every $l\geq q_{\mathrm d}$, write
$l=kq_{\mathrm d}+j,
k\geq1,
j\in\{0,\ldots,q_{\mathrm d}-1\}$.
Since $M_{\mathrm d}\geq1$, it follows that
\[
|F_{\mathrm d}^{l}|
\leq
|F_{\mathrm d}^{q_{\mathrm d}}|^{k}
|F_{\mathrm d}^{j}|
\leq
q_{\mathrm d}
M_{\mathrm d}^{2q_{\mathrm d}-2}
c_{\mathit{\Theta}_{\mathrm d}}
\varepsilon_{\mathrm d},
\]
and
\[
\sum_{i=0}^{\infty}|F_{\mathrm d}^{i}|
\leq
\sum_{k=0}^{\infty}
|F_{\mathrm d}^{q_{\mathrm d}}|^{k}
\sum_{j=0}^{q_{\mathrm d}-1}M_{\mathrm d}^{j}
\leq
2\sum_{j=0}^{q_{\mathrm d}-1}M_{\mathrm d}^{j}.
\]

Let $\bar{\eta}_{\mathrm d}$ and
$\bar{\alpha}_{\mathrm d1}$ be computable uniform upper bounds
on $\eta_{\mathrm d}$ and $\alpha_{\mathrm d1}$, respectively,
over $d_{0}\in\mathbb{I}_{[0,D_{M}]}$,
$|\mathit{\Delta}B_{\mathrm d}|
\leq
1/c_{\mathit{\Theta}_{\mathrm d}}$.
Then, for every $l\geq q_{\mathrm d}$,
\[
\rho_{\mathrm d0}(l)
=
\eta_{\mathrm d}|F_{\mathrm d}^{l}|
\leq
\bar{\eta}_{\mathrm d}
q_{\mathrm d}
M_{\mathrm d}^{2q_{\mathrm d}-2}
c_{\mathit{\Theta}_{\mathrm d}}
\varepsilon_{\mathrm d},
\]
and
\[
\begin{aligned}
\rho_{\mathrm d1}
\leq{}
2\bar{\eta}_{\mathrm d}
\bar{\alpha}_{\mathrm d1}
\sum_{j=0}^{q_{\mathrm d}-1}M_{\mathrm d}^{j}+
\sum_{i=0}^{3h-1}|A_{\mathrm d}^{i}|
+
\bar{\alpha}_{\mathrm d1}.
\end{aligned}
\]

Therefore, it is sufficient to choose
\[
0\leq\varepsilon_{\mathrm d}
<
\min\left\{
\frac{1}{
2q_{\mathrm d}
M_{\mathrm d}^{q_{\mathrm d}-1}
c_{\mathit{\Theta}_{\mathrm d}}
},
\frac{1}{
\bar{\eta}_{\mathrm d}
q_{\mathrm d}
M_{\mathrm d}^{2q_{\mathrm d}-2}
c_{\mathit{\Theta}_{\mathrm d}}
}
\right\}
\]
and
\[
0\leq f_{\mathrm d0}
<
\frac{
1-
\bar{\eta}_{\mathrm d}
q_{\mathrm d}
M_{\mathrm d}^{2q_{\mathrm d}-2}
c_{\mathit{\Theta}_{\mathrm d}}
\varepsilon_{\mathrm d}
}{
2\bar{\eta}_{\mathrm d}
\bar{\alpha}_{\mathrm d1}
\displaystyle\sum_{j=0}^{q_{\mathrm d}-1}M_{\mathrm d}^{j}
+
\displaystyle\sum_{i=0}^{3h-1}|A_{\mathrm d}^{i}|
+
\bar{\alpha}_{\mathrm d1}
}.
\]
The first bound on $\varepsilon_{\mathrm d}$ also guarantees
$\varepsilon_{\mathrm d}
<1/c_{\mathit{\Theta}_{\mathrm d}}$.
Hence, these bounds ensure that
\[
\rho_{\mathrm d}(l)
=
\rho_{\mathrm d0}(l)
+
\rho_{\mathrm d1}f_{\mathrm d0}
<1,
\qquad
l\geq q_{\mathrm d}.
\]
All the quantities in these bounds are computable from the
known system matrices, the delay bound $D_{M}$, and the
artificial delay $h$.
\end{remark}

Denote%
\begin{equation}
\delta_{\mathrm d}=\left(\eta_{\mathrm d}\sup_{l\in\mathbb N}\sum_{i=0}^{l-1}\left|F_{\mathrm d}^{\,i}\right|+1\right)(\alpha_{\mathrm d1}+\alpha_{\mathrm d2})+\sum_{i=0}^{3h-1}\left|A_{\mathrm d}^{i}\right|,
\label{delta_dd}
\end{equation}
where
\begin{equation}
\alpha_{\mathrm{d}2}=\sup_{k\in \mathbb{I}_{[2h,3h-1]}}%
{\displaystyle\sum_{i=2h}^{k}}
\left\vert A_{\mathrm{d}}^{k-i}B_{\mathrm{d}}B_{\mathrm{d}0}^{\mathrm{T}%
}(A_{\mathrm{d}}^{3h-i-1})^{\mathrm{T}}W_{\mathrm{d}}^{-1}%
A_{\mathrm{d}}^{5h-i}\right\vert . \label{alpha_2}%
\end{equation}
Similar to Remark~3, for sufficiently small $\mathit{\Delta}B_{\mathrm{d}}$,
the matrix \(F_{\mathrm d}\) is Schur stable. Consequently,
\(F_{\mathrm d}^{\,l}\to 0\) as \(l\to\infty\), and
$\sup_{l\in\mathbb{N}}
\sum_{i=0}^{l-1}
\left|F_{\mathrm d}^{\,i}\right|
<\infty$.
Therefore, \(\delta_{\mathrm d}\) defined in
\eqref{delta_dd} is finite.

Following \cite{Mazenc24css}, and similarly to the continuous-time case, a
discrete-time ISS estimate is called an almost fixed-time ISS estimate if the
time step after which it applies is uniform with respect to the initial
condition and if its guaranteed exponential convergence rate tends to
infinity as the uncertainty bounds tend to zero.

We are in a position to state the following result:

\begin{theorem}
\label{the6}Consider system (\ref{dis_sys_2}) under Assumptions \ref{assu_3}%
-\ref{ass4}. Given $h\geq\max\{D_{M},n\}$, consider the following controller
\begin{equation}
u\left(  k\right)  =-K_{\mathrm{d}}(k)p\left(  k\right)  y\left(  k-h\right)
, \label{d_uu}%
\end{equation}
where $l=\lfloor k/(3h)\rfloor$ and
\[
K_{\mathrm{d}}(k)
=B_{\mathrm{d}0}^{\mathrm{T}}
(A_{\mathrm{d}}^{3hl+3h-k-1})^{\mathrm{T}}
W_{\mathrm{d}}^{-1}
A_{\mathrm{d}}^{5h+3hl-k}.
\] Then the solution of system~(\ref{dis_sys_2}) under
controller~(\ref{d_uu}) satisfies
\begin{align}
|x(k)|
&\leq |x|_{[0,3(q_{\mathrm{d}}+1)h]}
\mathrm{e}^{\frac{\ln(\rho_{\mathrm{d}}(q_{\mathrm{d}}))}
{3(q_{\mathrm{d}}+1)h}(k-3(q_{\mathrm{d}}+1)h)}\nonumber\\
&\quad+\frac{\delta_{\mathrm{d}}(
\vert w_{\mathrm{d}1}\vert _{[0,k]}
+\vert w_{\mathrm{d}2}\vert _{[0,k]})}
{(1-\rho_{\mathrm{d}}(q_{\mathrm{d}}))^{2}},
\label{x_k_1}
\end{align}
for all $k\geq3(q_{\mathrm{d}}+1)h$, which establishes almost fixed-time ISS in the sense defined above. In particular, if $\mathit{\Delta}B_{\mathrm{d}}=0$ and
$f_{\mathrm{d}0}=0$, we have $|x(k)|\leq\delta_{\mathrm{d}}(\vert w_{\mathrm{d}1}\vert _{[k-3(q_{\mathrm{d}}+1)h,k]}
+\vert w_{\mathrm{d}2}\vert _{[k-3(q_{\mathrm{d}}+1)h,k]})$ for $k\geq3(q_{\mathrm{d}}+1)h$, which establishes fixed-time ISS in the
sense of Definition~\ref{def_discrete_ftiss}.
\end{theorem}

For readability, we place the proof in Appendix A2.

\subsection{Single Integrator with Unknown Fast-Varying Delay}

Consider the following single integrator with unknown fast-varying delay
\begin{equation}
\left\{
\begin{array}
[c]{rl}%
x(k+1) & =x(k)+B_{\mathrm{d}}u(k)+f_{\mathrm{d}}\left(  k,x(k)\right)
+w_{\mathrm{d}1}(k),\\[1.5mm]%
y\left(  k\right)  & =x\left(  k-d(k)\right)  +w_{\mathrm{d}2}(k),
\end{array}
\right.  \label{scalar_d}%
\end{equation}
where $x(k)\in\mathbb{R}$, $B_{\mathrm{d}}\in\mathbb{R}$, $d(k)\in
\mathbb{I}_{[0,{D}_{{{M}}}]}$ is an unknown fast-varying delay. Similar to
Proposition \ref{corl1}, we have the following conclusion.

\begin{proposition}
\label{corl2}Consider system (\ref{scalar_d}) under Assumptions \ref{assu_3}%
-\ref{ass4}. Given $h\geq\max\{D_{M},1\}$, consider the controller
(\ref{d_uu}) with $A_{\mathrm{d}}=1$. Then the solution of
system~(\ref{scalar_d}) under controller~(\ref{d_uu}) satisfies
(\ref{x_k_1}).
\end{proposition}

The proof follows similar steps as in Proposition \ref{corl1} and is omitted
for brevity.

\section{\label{sec4}Simulation Results}

\subsection{Example 1: Continuous-Time Case}

Consider the system borrowed from \cite{Mazenc25acc,Mazenc25tac}
\begin{equation}
\left\{
\begin{array}
[c]{rl}%
\dot{x}_{1}(t)= & x_{2}(t),\\
\dot{x}_{2}(t)= & (1+\delta_{\mathrm{b}})u(t)-\epsilon(t)x_{2}(t)+\frac
{G\delta_{\mathrm{b}}\sin(x_{1}(t))}{L},\\
y(t)= & x(t-\tau_{0})+w(t),
\end{array}
\right.  \label{sys_simu}%
\end{equation}
where $G=9.81$ is the gravitational constant, $L=1$ is the pendulum length,
$\epsilon(t)\in\lbrack0,\bar{\epsilon}]$ is an unknown
continuous function with $\bar{\epsilon}>0$ being a known constant,
$\delta_{\mathrm{b}}$ is an unknown bounded constant. Clearly,
$\mathit{\Delta}B=[0,\delta_{\mathrm{b}}]^{\mathrm{T}}$,
\[
A=\begin{bmatrix}0&1\\0&0\end{bmatrix},
\qquad
B_{0}=\begin{bmatrix}0&1\end{bmatrix}^{\mathrm T},
\]
and
$f(t,x)=[0,-\epsilon(t)x_{2}(t)+G\delta_{\mathrm{b}}
\sin(x_{1}(t))/L]^{\mathrm{T}}$.
Similar to \cite{Mazenc25tac}, take the initial condition
$x(0)=[-1,1]^{\mathrm{T}}$. Similar to \cite{Wang24auto,Zime25auto},
take $w(t)=0.01(\mathrm{rnd}(2)+[\sin(11t),\sin(12t)]^{\mathrm{T}})$, where
$\mathrm{rnd}(n)$ denotes uniformly distributed in the interval $[0,1]^{n}$
random numbers. For simulation, we take $S(t)=\sin^2(\pi t/h)$ and consider
five cases. Case I: $\delta_{\mathrm b}=0$, $f(t,x)=0$, and
$(\tau_0,h)=(0.2,0.25)$; Case II: $\delta_{\mathrm b}=0.25$,
$\epsilon(t)=0.540591$ (the same as in \cite{Mazenc25tac}), and
$(\tau_0,h)=(0.2,0.25)$; and Cases III--V:
$\delta_{\mathrm b}=0.001$ and $\epsilon(t)=0.540591$, with
$(\tau_0,h)=(2,2)$, $(0.2,0.25)$, and $(0.2,2)$, respectively.
Thus, Cases IV and V differ only in $h$.

For the considered system, $q=2$, and the uncertainty bounds can
be taken as
\[
\varepsilon=\lvert\delta_{\mathrm{b}}\rvert,
\qquad
f_{0}=\sqrt{\bar{\epsilon}^{2}
+\left(G\delta_{\mathrm{b}}/L\right)^{2}}.
\]
Taking $\tau_{M}=h$ and using the computable estimates in
Remark~\ref{remark_1}, for $h=0.25$ one obtains the sufficient region
\[
\begin{aligned}
0\leq\varepsilon &<4.40992\times10^{-4},\\
0\leq f_{0} &<
3.22564\times10^{-4}
\left(
1-\frac{\varepsilon}{4.40992\times10^{-4}}
\right).
\end{aligned}
\]
Thus, Case I, for which $\varepsilon=f_{0}=0$, satisfies these
sufficient bounds. In Case II, the actual uncertainty levels are
\[
\varepsilon=0.25,
\qquad
f_{0}
=
\sqrt{0.540591^{2}+(9.81\times0.25)^{2}}
=
2.51137,
\]
which lie outside the above sufficient region. Case IV also lies outside
this sufficient region, since $\varepsilon=0.001$ and
$f_0=0.540680$. For $h=2$, the
corresponding sufficient region is
\[
\begin{aligned}
0\leq\varepsilon &<1.59832\times10^{-7},\\
0\leq f_{0} &<
4.47715\times10^{-11}
\left(
1-\frac{\varepsilon}{1.59832\times10^{-7}}
\right).
\end{aligned}
\]
In Cases III and V, the actual uncertainty levels are
$\varepsilon=0.001$ and
\[
f_{0}
=
\sqrt{0.540591^{2}+(9.81\times0.001)^{2}}
=
0.540680,
\]
which also lie outside the corresponding sufficient region.

The state trajectories are shown in Fig.~\ref{figure1}. According to the
theoretical estimate, the convergence time in Case I is $3qh=3\times
2\times0.25=1.5$. Cases IV and V show that, under the considered
uncertainty, a smaller $h$ yields faster convergence in the noise-free case.
Note that for $h\to 0^+$ the robustness may become worse since the Gramian $W_{\mathrm{c}}$
may become ill-conditioned and the controller depends on its inverse. Moreover, the
presence of uncertainties increases the convergence time. Fig. \ref{figure1}
also indicates that the closed-loop system is robust against the measurement
noise. Although the uncertainty levels in Cases II--V lie
outside the above computable sufficient regions, the simulation results show
that the closed-loop system remains stable and robust. This indicates that the
computable theoretical bounds are conservative.

\begin{figure}[t]
\centering \includegraphics[scale=0.6]{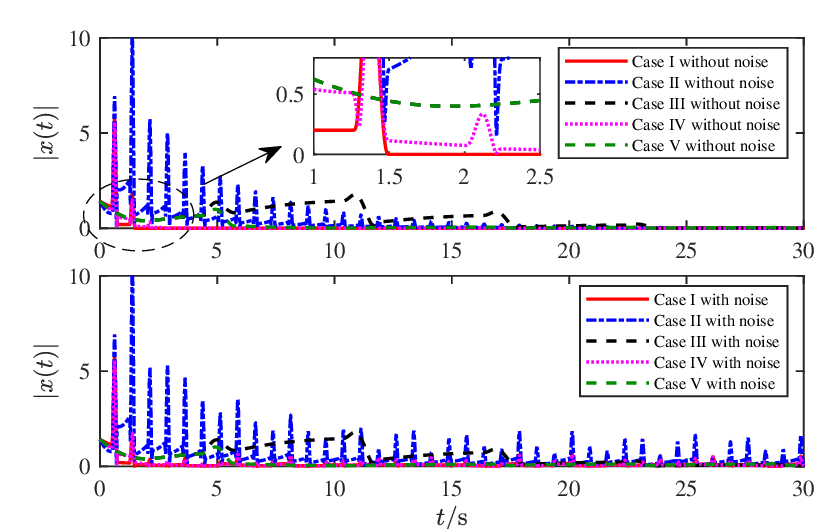}\caption{The state responses for
system (\ref{sys_simu}) without/with noise (top/bottom)}%
\label{figure1}%
\end{figure}

\subsection{Example 2: Discrete-Time Case}

Consider the discretized system corresponding to system (\ref{sys_simu})
\begin{equation}
\left\{
\begin{array}
[c]{rl}%
x\left(  k+1\right)  = & A_{\mathrm{d}}x(k)  +(B_{\mathrm{d}%
0}+\mathit{\Delta}B_{\mathrm{d}})u(  k)  +f_{\mathrm{d}}(k,x),\\
y(k)= & x(k-r)+w_{\mathrm{d}}(k),
\end{array}
\right.  \label{x_ky}%
\end{equation}
where $\mathit{\Delta}B_{\mathrm{d}}=\delta_{\mathrm{bd}}B_{\mathrm{d}0}$,
$B_{\mathrm{d}0}=[T_{\mathrm{s}}^{2}/2,T_{\mathrm{s}}]^{\mathrm{T}}$, and
$f_{\mathrm{d}}(k,x)=-
[T_{\mathrm{s}}^{2}/2,T_{\mathrm{s}}]^{\mathrm{T}}
(\epsilon_{\mathrm{d}}(k)x_{2}(k)
-G\delta_{\mathrm{bd}}\sin(x_{1}(k))/L)$, and
$A_{\mathrm{d}}=\left[
\begin{array}{@{}cc@{}}
1&T_{\mathrm{s}}\\[-0.8mm]
0&1
\end{array}
\right]$,
with $T_{\mathrm{s}}$ being the sampling period. Similar to \cite{Mazenc25tac}%
, take the initial condition $x(0)=[-1,1]^{\mathrm{T}}$.
Here we take $h=2$, $T_{\mathrm{s}}=0.1$, and $w_{\mathrm{d}}%
(k)=0.01(\mathrm{rnd}(2)+[\sin(11kT_{\mathrm{s}}),\sin(12kT_{\mathrm{s}%
})]^{\mathrm{T}})$. We consider three cases. Case I:
$r=1\leq h,\epsilon_{\mathrm{d}}(k)=0,\delta_{\mathrm{bd}}=0$; Case II:
$r=1\leq h,\epsilon_{\mathrm{d}}(k)=0.540591,\delta_{\mathrm{bd}}=0.1$;
Case III:
$r=2\leq h,\epsilon_{\mathrm{d}}(k)=0.540591,\delta_{\mathrm{bd}}=0.1$.

For the considered discrete-time system, $q_{\mathrm{d}}=2$ and
\[
\left\lvert B_{\mathrm{d}0}\right\rvert
=\sqrt{(T_{\mathrm{s}}^{2}/2)^{2}+T_{\mathrm{s}}^{2}}
=0.100125.
\]
Let
$\bar{\epsilon}_{\mathrm{d}}=\sup_{k\in\mathbb{N}}
\lvert\epsilon_{\mathrm{d}}(k)\rvert$. The uncertainty bounds can be taken as
\[
\begin{aligned}
\varepsilon_{\mathrm{d}}
&=\lvert\delta_{\mathrm{bd}}\rvert
\left\lvert B_{\mathrm{d}0}\right\rvert,\\
f_{\mathrm{d}0}
&=\left\lvert B_{\mathrm{d}0}\right\rvert
\sqrt{\bar{\epsilon}_{\mathrm{d}}^{2}
+\left(G\delta_{\mathrm{bd}}/L\right)^{2}}.
\end{aligned}
\]
Taking $D_{M}=h=2$ and using the computable estimates in the
discrete-time counterpart of Remark~\ref{remark_1}, one obtains the sufficient
region
\[
\begin{aligned}
0\leq\varepsilon_{\mathrm{d}}
&<1.07070\times10^{-4},\\
0\leq f_{\mathrm{d}0}
&<
2.56251\times10^{-4}
\left(
1-\frac{\varepsilon_{\mathrm{d}}}{1.07070\times10^{-4}}
\right).
\end{aligned}
\]
Thus, Case I, for which
$\varepsilon_{\mathrm{d}}=f_{\mathrm{d}0}=0$, satisfies these sufficient
bounds. In Cases II and III, the actual uncertainty levels are
\[
\begin{aligned}
\varepsilon_{\mathrm{d}}
&=0.1\times0.100125=0.0100125,\\
f_{\mathrm{d}0}
&=0.100125
\sqrt{0.540591^{2}+(9.81\times0.1)^{2}}
=0.112149,
\end{aligned}
\]
which lie outside the above sufficient region.

The state trajectories are shown in Fig. \ref{figure2}. According to the
theoretical estimate, the convergence time in Case I is $3q_{\mathrm d}h=3\times
2\times2=12$. Fig. \ref{figure2} shows that the closed-loop system remains
robust in the presence of measurement noise. Although the
uncertainty levels in Cases II and III lie outside the above computable
sufficient region, the simulation results show that the closed-loop system
remains stable and robust. This indicates that the computable discrete-time
bounds are conservative.

\begin{figure}[t]
\centering \includegraphics[scale=0.6]{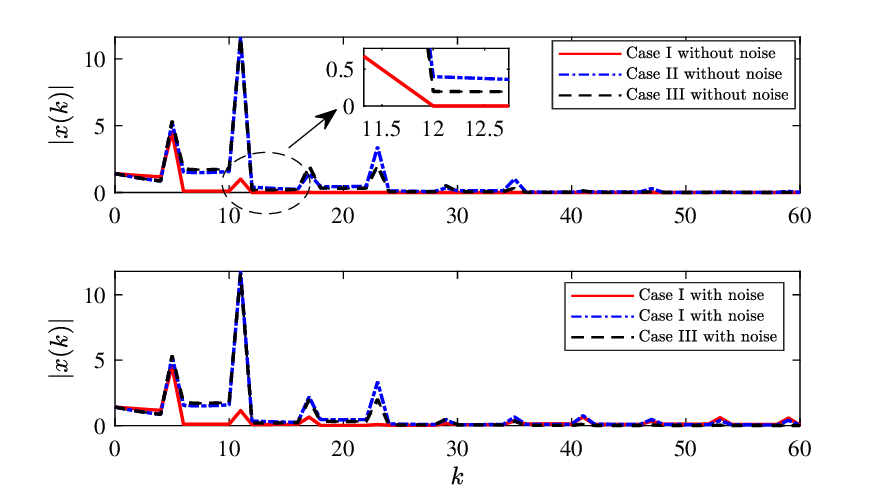}\caption{The state responses for
system (\ref{x_ky}) without/with noise (top/bottom)}%
\label{figure2}%
\end{figure}
\subsection{Example 3: Triple Integrator System}

Consider a triple integrator system (\ref{sys_norm}) with
\[
A=
\begin{bmatrix}
0 & 1 & 0\\
0 & 0 & 1\\
0 & 0 & 0
\end{bmatrix},
\qquad
B_{0}=
\begin{bmatrix}
0\\
0\\
1
\end{bmatrix}.
\]
For this system, we can use Theorem~\ref{the_4} to design the controller.
In the noise-free case, similar to \cite{Zime25auto}, take
$x(0)=[1,0,0]^{\mathrm{T}}$. The state trajectory is shown by the red line in the top panel of Fig.
\ref{figure3}. It can be observed that in the presence of unknown measurement
delay bounded by $h=5$ and the delay $\tau_{0}=5$ chosen for simulations, our
closed-loop system converges exactly to zero in time $15$ in the noise-free
case. Note that in the noise-free case with essentially smaller unknown input
delay $0.05$, the hyperexponential controller in \cite{Zime25auto} and the
finite-time controller in \cite{Polyakov11tac} achieve ultimate bounds of $0.01$ and $0.06$ respectively (see Fig. 4 in \cite{Zime25auto}). It is worth noting here that for a general initial condition $x(0)$, the exact convergence time is $3qh = 45$; however, for the specific initial condition $x(0) = [a,0,0]^{\mathrm{T}}$, where $a$ is an arbitrary constant, according to (\ref{x_3lh00}), we have $x(3h) = 0$, i.e., at $15$. For the noisy case with this specific initial condition, it follows from (\ref{x_3lh_1}) that the contribution of the initial state vanishes after $3h$.

In the presence of a small measurement noise $w(t)=0.01v(t)$ with
$v(t)=0.01(\mathrm{rnd}(3)+[\sin(11t),\sin(12t),$ $\sin(13t)]^{\mathrm{T}})$ used
in \cite{Zime25auto}, our controller with delay $\tau
_{0}=5$ leads to an ultimate bound of $0.06$ over the final time interval $[15,200]$ (see the black line in the top panel of Fig.
\ref{figure3}). Slightly larger ultimate bounds without delay, but with $w(t)=v(t)$ are achieved by
\cite{Zime25auto} (ultimate bound of $0.10$) and \cite{Polyakov11tac}
(ultimate bound of $0.13$) over the final time interval $[6,15]$ (see Fig. 3 in \cite{Zime25auto}). It is seen that without delay, the
controllers of \cite{Zime25auto} and \cite{Polyakov11tac} lead to smaller
ultimate bounds, but our method is superior for large unknown delays.

In the presence of a measurement noise $w(t)=v(t)$, which is used in
\cite{Zime25auto}, our controller with delay $\tau_{0}=5$ leads to an ultimate
bound of $6$ over the final time interval $[15,200]$ for initial conditions
$x(0)=[1,0,0]^{\mathrm{T}}$, $x(0)=[5,0,0]^{\mathrm{T}}$ and $x(0)=[10,0,0]^{\mathrm{T}}$ (see the bottom panel of Fig.
\ref{figure3}). From
(\ref{ni_A_1}), the contribution of these specific initial states vanishes after $t=15$, so a larger ultimate bound after $t=15$ here is caused by the larger measurement-noise
magnitude in the disturbance term of (\ref{ni_A_1}).

\begin{figure}[t]
\centering \includegraphics[scale=0.6]{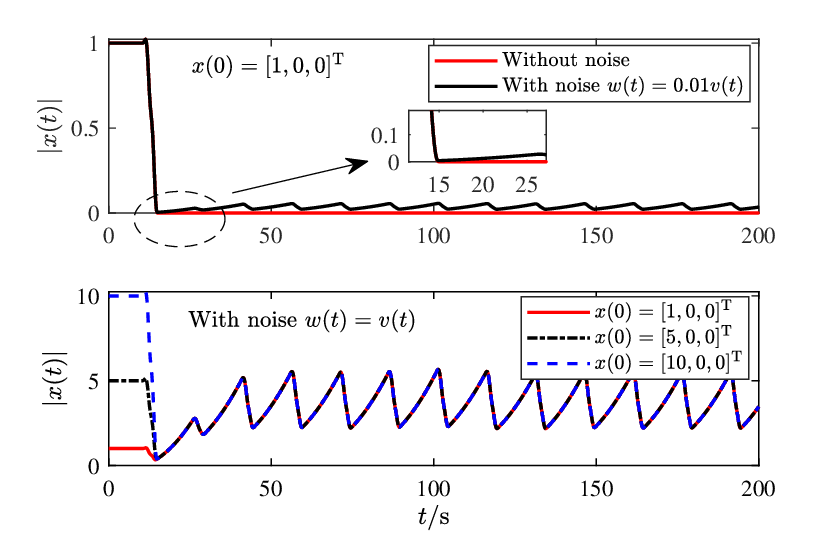}\caption{The state responses for
triple integrator system}%
\label{figure3}%
\end{figure}
\subsection{Example 4: Single Integrator System}

Consider a single integrator system (\ref{scalar}) with $f(t,x)=0$, $B=1$ and $w_{1}(t)=0$.
For convenience, we denote $w_2(t)$ as $w(t)$. For this system, we can use
Proposition~\ref{corl1} to design the controller. In the noise-free case, take
three different initial conditions $x(0)=1$, $x(0)=5$ and $x(0)=10$. The state trajectories are shown in the top panel of Fig.
\ref{figure4}. It can be observed that in the presence of unknown
fast-varying measurement delay bounded by $h=5$ and the fast-varying delay $\tau(t)=4+\sin(20t)$ chosen for simulations, our
closed-loop system converges exactly to zero in time $15$ in the noise-free
case.

In the presence of a measurement noise $w(t)=100v_{3}(t)$, where $v_{3}(t)$  is the third element of $v(t)$ in the previous example, our controller with fast-varying delay $\tau(t)=4+\sin(20t)$ leads to an ultimate
bound of $0.6$ over the final time interval $[15,50]$ for initial conditions
$x(0)=1$, $x(0)=5$ and $x(0)=10$ (see the bottom panel of Fig. \ref{figure4}).
\begin{figure}[t]
\centering \includegraphics[scale=0.6]{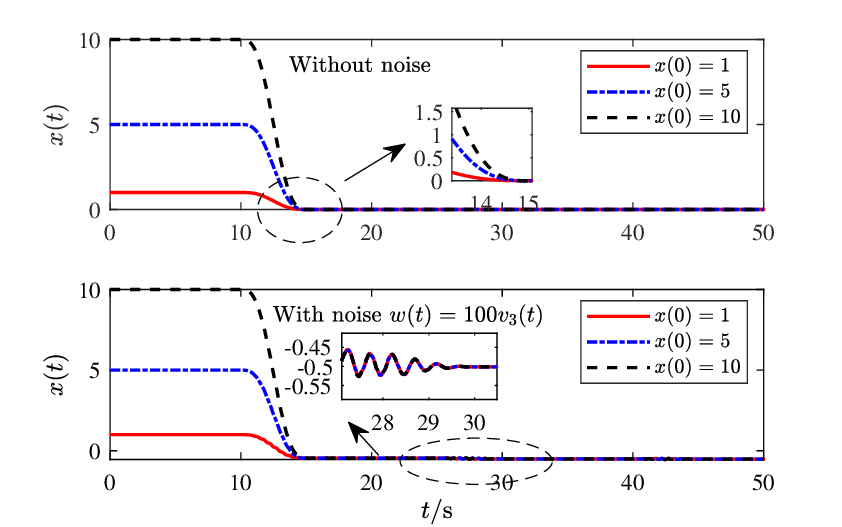}\caption{The state responses for
single integrator system}%
\label{figure4}%
\end{figure}

\section{\label{sec5}Conclusion}

This paper has studied fixed-time input-to-state stabilization using artificial delay for a class of continuous-time
and discrete-time systems with unknown constant delay
in the full-state measurement or input. The proposed approach can
handle unknown large constant delays and achieve fixed-time stabilization,
which is not easily attainable with conventional methods. Numerical examples
have demonstrated its effectiveness. When the sampling interval is sufficiently small, the difference between the continuous-time controller and the sampled-data controller can be treated as a small perturbation. Therefore, the proposed method can also be extended to the sampled-data control framework as studied in~\cite{Mazenc25tac}.

\section*{Appendix}

\subsection*{A1: The Proof of Theorem \ref{the_4}}

Denote $\lambda_{1}(t)=w_{1}(t)+f(t,x)$. For $t\in\lbrack3lh,(3l+2)h]$, $l\in%
\mathbb{Z}
_{0}^{+}$, $g(t)=0$. Then system~(\ref{sss1}) under controller~(\ref{vv})
can be represented as
$\dot{x}(t)=Ax(t)+\lambda_{1}(t)$,
whose solution is
given by
\begin{equation}
x(t)=\mathrm{e}^{A(t-3lh)}x(3lh)+\int_{3lh}^{t}\mathrm{e}^{A\left(
t-s\right)  }\lambda_{1}(s)\mathrm{d}s. \label{32}%
\end{equation}
For $t\in((3l+2)h,(3l+3)h]$, $l\in%
\mathbb{Z}
_{0}^{+}$, we have $g(t)=1$, and the closed-loop system is expressed as
\begin{align}
&\dot{x}(t)=  Ax(t)-B_{0}S(t)K(
t)  x(t-h-\tau_{0})+\lambda_{1}(t) \nonumber\\
&  -BS(t)K(t) w_{2}(t-h)-\mathit{\Delta}BS(t)K(t)
x(t-h-\tau_{0})  . \label{33}%
\end{align}
By using the fact that $t-h-\tau_{0}\in(3lh,(3l+2)h]$, substituting (\ref{32})
into (\ref{33}), and employing the variation of constants formula, we have the
solution to
\begin{equation}
x\left(  t\right)  =\left(  \mathrm{e}^{A\left(  t-3lh\right)  }%
\mathit{\Delta}(t)+\mathit{\Theta}(t)\right)  x\left(  3lh\right)
+\mathrm{e}^{A\left(  t-3lh\right)  }\mathit{\Pi}_{w}\left(  t\right)  ,
\label{34}%
\end{equation}
where $\mathit{\Theta}(t)$ is defined in (\ref{ceta_c}) and $\mathit{\Delta
}(t)$ is defined in (\ref{delta}), and
\begin{align*}
&\mathit{\Pi}_{w}(t)=  -\int_{(3l+2)h}^{t}\mathrm{e}^{A(3lh-s)}BS(s)K\left(  s\right)w_{2}(s-h)\mathrm{d}s\\
&  +\int_{3lh}^{t}\mathrm{e}^{A(3lh-s)}\lambda_{1}(s)\mathrm{d}s -\int_{(3l+2)h}^{t}\mathrm{e}^{A(3lh-s)}BS(s)K\left(  s\right) \\
&  \times\int_{3lh}^{s-h-\tau_{0}}\mathrm{e}^{A(s-h-\tau_{0}-r)}\lambda
_{1}(r)\mathrm{d}r\mathrm{d}s.
\end{align*}
From definitions of $\alpha_{1}$ and $\alpha_{2}$ in (\ref{alpha_01}) and
(\ref{alpha_02}), we can get
\[
\left|\mathit{\Pi}_{w}(t)\right|
\leq
\alpha_{1}\left|\lambda_{1}\right|_{[3lh,3(l+1)h]}
+
\alpha_{2}\left|w_{2}\right|_{[3lh,3(l+1)h]}.
\]

From (\ref{34}) and $\mathit{\Delta}((3l+3)h)=I_{n}-\mathrm{e}^{-A\tau_{0}}$,
it follows
\begin{equation}
x\left( (3l+3)h\right)  =Fx\left(  3lh\right)  +\mathrm{e}^{3Ah}\mathit{\Pi
}_{w}\left(  (3l+3)h\right)  . \label{35}%
\end{equation}
where $F$ is defined in (\ref{F_def}).

Based on the above analysis, we can conclude that for $l\geq1,$
\begin{align}
  &\vert x(3lh) \vert \leq  \vert F^{l}x(
0) \vert +\sum_{i=1}^{l}\left \vert F^{l-i}\, \mathrm{e}^{3Ah}\right \vert\nonumber \\
&   \times\left(
\alpha_{1}\left \vert \lambda_{1}\right \vert _{[3(i-1)h,3ih]}+\alpha
_{2}\left\vert w_{2}\right\vert _{[3(i-1)h,3ih]}\right)  .\label{x_3lh_1}%
\end{align}
Then for $t\in\lbrack3lh,(3l+2)h],$
\begin{align}
\left\vert x\left(  t\right)  \right\vert \leq &  \left\vert \mathrm{e}%
^{A\left(  t-3lh\right)  }\right\vert \left\vert x\left(  3lh\right)
\right\vert +\int_{3lh}^{t}\left\vert \mathrm{e}^{A\left(  t-s\right)
}\right\vert \mathrm{d}s\left\vert \lambda_{1}\right\vert _{[3lh,t]}%
\nonumber\\
\leq &  \left\vert \mathrm{e}^{A\left(  t-3lh\right)  }\right\vert \left\vert
F^{l}\right\vert \left\vert x\left(  0\right)  \right\vert +\left\vert \mathrm{e}^{A\left(  t-3lh\right)  }\right\vert \sum_{i=1}%
^{l}\left\vert F^{l-i}\,\mathrm{e}^{3Ah}\right\vert \nonumber\\
&  \times\left(  \alpha_{1}\left\vert \lambda_{1}\right\vert _{[3(i-1)h,3ih]}%
+\alpha_{2}\left\vert w_{2}\right\vert _{[3(i-1)h,3ih]}\right) \nonumber\\
&  +\int_{3lh}^{t}\left\vert \mathrm{e}^{A\left(  t-s\right)  }\right\vert
\mathrm{d}s\left\vert \lambda_{1}\right\vert _{[3lh,t]}, \label{x_11}%
\end{align}
and for $t\in\lbrack(3l+2)h,(3l+3)h],$
\begin{align}
&\vert x(t)\vert  \leq\mathrm{e}^{\vert
A\vert(t-3lh)}\vert \mathit{\Delta}(t)\vert
\vert x(3lh)\vert +\mathrm{e}^{\vert
A\vert(t-3lh)}\vert \mathit{\Pi}_{w}(
t)\vert \nonumber\\
&  +\vert \mathit{\Theta}(t)\vert \vert x(3lh)
\vert \nonumber\\
&  \leq(\mathrm{e}^{3\vert A\vert h}\vert
\mathit{\Delta}(t)\vert +\vert \mathit{\Theta}(t)\vert
) \left(\vert F^{l}\vert \vert x(0)
\vert\right. \nonumber\\
&  +\left.  \sum_{i=1}^{l}\left\vert F^{\,l-i}\,\mathrm{e}^{3Ah}%
\right\vert \left(  \alpha_{1}\left\vert \lambda_{1}\right\vert
_{[3(i-1)h,3ih]}+\alpha_{2}\left\vert w_{2}\right\vert _{[3(i-1)h,3ih]}\right)  \right) \nonumber\\
&  +\mathrm{e}^{\left\vert A\right\vert \left(  t-3lh\right)  }\left(
\alpha_{1}\left\vert \lambda_{1}\right\vert _{[3lh,t]}+\alpha_{2}\left\vert w_{2}\right\vert _{[3lh,t]
}\right)  . \label{x_22}%
\end{align}

By restarting the above estimates from the cycle boundary $3(l-q)h\in[t-3(q+1)h,t]$ and using the periodicity of the controller, the same estimates apply to this moving interval. From (\ref{delta_c}), (\ref{x_11}), and (\ref{x_22}), we can get
\begin{align}
\left\vert x(t)\right\vert
&\leq\rho(q)\left\vert x\right\vert
_{[t-3(q+1)h,t]}\nonumber\\
&\quad+\delta\left(
\left\vert w_{1}\right\vert _{[t-3(q+1)h,t]}
+\left\vert w_{2}\right\vert _{[t-3(q+1)h,t]}\right),
\end{align}
where $t\geq3(q+1)h$. Here, denote
$V(t)=x(t+3(q+1)h)$, it follows that
\begin{align}
\left\vert V(t)\right\vert
&\leq\rho(q)\left\vert V\right\vert
_{[t-3(q+1)h,t]}\nonumber\\
&\quad+\delta\left({
\left\vert w_{1}\right\vert _{[t,t+3(q+1)h]}
+\left\vert w_{2}\right\vert _{[t,t+3(q+1)h]}}\right),\quad t\geq0.
\end{align}
According to $\rho(q)\in(0,1)$ and Lemma 1 in
\cite{Mazen14tac}, we can obtain that
\begin{align}
\left\vert V(t)\right\vert
&\leq\left\vert V\right\vert_{[-3(q+1)h,0]}
\mathrm{e}^{\frac{\ln\left({\rho(q)}\right)t}
{3(q+1)h}}\nonumber\\
&\quad+\frac{\delta\left({
\left\vert w_{1}\right\vert _{[0,t+3(q+1)h]}
+\left\vert w_{2}\right\vert _{[0,t+3(q+1)h]}}\right)}
{\left(1-\rho(q)\right)^{2}}.
\end{align}
If $\mathit{\Delta}B=0$ and $f_{0}=0$, then $F^{q}=0$ and $\rho(q)=0$. Hence, the preceding delay inequality directly gives $|x(t)|\leq\delta(|w_{1}|_{[t-3(q+1)h,t]}+|w_{2}|_{[t-3(q+1)h,t]})$ for $t\geq3(q+1)h$.
Therefore, (\ref{ni_A_1}) is proved.\hfill$\blacksquare$

\subsection*{A2: The Proof of Theorem \ref{the6}}

Denote $\lambda_{\mathrm{d}1}(k)=w_{\mathrm{d}1}(k)+f_{\mathrm{d}}(k,x(k))$.
For $k\in\mathbb{I}_{[3hl,3hl+2h-1]}$, $l\in%
\mathbb{Z}
_{0}^{+}$, $p(k)=0$. Then system~(\ref{dis_sys_2}) under
controller~(\ref{d_uu}) can be represented as
$x(k+1)=A_{\mathrm{d}}x(k)+\lambda_{\mathrm{d}1}(k)$,
which implies that
\begin{equation}
x(k+1)=A_{\mathrm{d}}^{k+1-3hl}x(3hl)+%
{\displaystyle\sum_{i=3hl}^{k}}
A_{\mathrm{d}}^{k-i}\lambda_{\mathrm{d}1}(i). \label{x_k11}%
\end{equation}

For $k\in\mathbb{I}_{[3hl+2h,3hl+3h-1]}$, $l\in%
\mathbb{Z}
_{0}^{+}$, we have $p(k)=1$, and system~(\ref{dis_sys_2}) under
controller~(\ref{d_uu}) can be represented as
\begin{align}
  &x(k+1)=  A_{\mathrm{d}}x(k)-B_{\mathrm{d}0}K_{\mathrm{d}}(k)x(k-d_{0}%
-h)+\lambda_{\mathrm{d}1}(k)\nonumber\\
&  -\mathit{\Delta}B_{\mathrm{d}}K_{\mathrm{d}}(k)x(k-d_{0}-h) -B_{\mathrm{d}}K_{\mathrm{d}}(k)w_{\mathrm{d}%
2}(k-h). \label{x_k}%
\end{align}
By using the fact that $3hl\leq k-d_{0}-h\leq3hl+2h-1$, substituting
(\ref{x_k11}) into (\ref{x_k}) yields
\begin{equation}
x(k+1)=(  A_{\mathrm{d}}^{k+1-3hl}\mathit{\Delta}_{\mathrm{d}%
}(k)+\mathit{\Theta}_{\mathrm{d}}(k))  x(3hl)+\mathit{\Pi}_{\mathrm{d}%
w}(k), \label{x_k_1_2}%
\end{equation}
where $\mathit{\Delta}_{\mathrm{d}}(k)$ is defined in (\ref{delta_ds}) and
$\mathit{\Theta}_{\mathrm{d}}(k)$ is defined in (\ref{ceta_ds}), and
\begin{align*}
\mathit{\Pi}_{\mathrm{d}w}(k)=  &
{\displaystyle\sum_{i=3hl}^{k}}
A_{\mathrm{d}}^{k-i}\lambda_{\mathrm{d}1}(i)-%
{\displaystyle\sum_{i=3hl+2h}^{k}}
A_{\mathrm{d}}^{k-i}B_{\mathrm{d}}K_{\mathrm{d}}(i)\\
&  \times%
{\displaystyle\sum_{r=3hl}^{i-d_{0}-h-1}}
A_{\mathrm{d}}^{i-d_{0}-h-1-r}\lambda_{\mathrm{d}1}(r)\\
&  -%
{\displaystyle\sum_{i=3hl+2h}^{k}}
A_{\mathrm{d}}^{k-i}B_{\mathrm{d}}K_{\mathrm{d}}(i)w_{\mathrm{d}2}(i-h).
\end{align*}
It follows from (\ref{alpha_1}) and (\ref{alpha_2}) that
\begin{align}
\left|\mathit{\Pi}_{\mathrm{d}w}(k)\right|
\leq{}&
\alpha_{\mathrm{d}1}
\sup_{m\in[3hl,\,3(l+1)h-1]}
\left|\lambda_{\mathrm{d}1}(m)\right|
\nonumber\\
&+
\alpha_{\mathrm{d}2}
\left|w_{\mathrm{d}2}\right|_{[3hl,\,3(l+1)h-1]}.
\end{align}
For $k=(3l+3)h-1$, we have
\begin{equation}
x((3l+3)h)=F_{\mathrm{d}}x(3lh)+\mathit{\Pi_{\mathrm{d}w}}((3l+3)h-1),
\label{x_3lh}%
\end{equation}
where $F_{\mathrm{d}}$ is defined in (\ref{F_d_def}). In view of
$(A_{\mathrm{d}}-I_{n})^{q_{\mathrm{d}}}=0$ and $h-d_{0}\geq0$, we can get that $I_{n}-A_{\mathrm{d}}^{h-d_{0}}$ is nilpotent, namely,
$\mathit{\Delta}_{\mathrm{d}}^{q_{\mathrm{d}}}((3l+3)h-1)=(I_{n}-A_{\mathrm{d}}^{h-d_{0}%
})^{q_{\mathrm{d}}}=0$.

By iterating (\ref{x_3lh}),
we can get for $l\geq q_{\mathrm{d}}$,
\begin{align}
x\left(  3hl\right) =F_{\mathrm{d}}^{l}x\left(0\right)
+\sum_{i=1}^{l}F_{\mathrm{d}}^{l-i}\,\mathit{\Pi}_{\mathrm{d}w}\left(
3hi-1\right)  . \label{x_3hn}%
\end{align}
Based on the above analysis, we can conclude that for $l\geq1$, $\left\vert
x\left(  3hl\right)  \right\vert \leq g_{\mathrm{d}12}$, where
\begin{align*}
&  g_{\mathrm{d}12}\triangleq\vert F_{\mathrm{d}}\vert
^{l}\vert x(0)\vert +\sum_{i=1}^{l}\vert
F_{\mathrm{d}}^{\,l-i}\vert \\
&  \times\left(  \alpha_{\mathrm{d}1}\sup_{m\in\lbrack3h(i-1)
,3hi-1]}\vert \lambda_{\mathrm{d}1}(m)\vert +\alpha
_{\mathrm{d}2}\vert w_{\mathrm{d}2}\vert _{[3h(i-1),3hi-1]}\right)  .
\end{align*}
Therefore, for $k\in\mathbb{I}_{[3hl,3hl+2h-1]}$, the solution to
(\ref{x_k11}) satisfies
\begin{align*}
&  \left\vert x(k+1)\right\vert \leq\left\vert A_{\mathrm{d}}^{k+1-3hl}%
\right\vert \left\vert x\left(  3hl\right)  \right\vert +%
{\displaystyle\sum_{i=3hl}^{k}}
\left\vert A_{\mathrm{d}}^{k-i}\lambda_{\mathrm{d}1}(i)\right\vert \\
\leq  &  \left\vert A_{\mathrm{d}}^{k+1-3hl}\right\vert g_{\mathrm{d}12}+%
{\displaystyle\sum_{i=3hl}^{k}}
\left\vert A_{\mathrm{d}}^{k-i}\right\vert \sup_{m\in\lbrack3hl,k]}\left\vert
\lambda_{\mathrm{d}1}(m)\right\vert ,
\end{align*}
and for $k\in\mathbb{I}_{[3hl+2h,3hl+3h-1]}$, the solution to (\ref{x_k_1_2})
satisfies
\begin{align*}
&  |x(k+1)|\\
\leq &  \left(  \left\vert A_{\mathrm{d}}^{k+1-3hl}\mathit{\Delta}%
_{\mathrm{d}}\left(  k\right)  \right\vert +\left\vert \mathit{\Theta
}_{\mathrm{d}}(k)\right\vert \right)  \left\vert x\left(  3hl\right)
\right\vert +\left\vert \mathit{\Pi}_{\mathrm{d}w}\left(  k\right)
\right\vert \\
\leq &  \left\vert A_{\mathrm{d}}^{k+1-3hl}\mathit{\Delta}_{\mathrm{d}}\left(
k\right)  \right\vert g_{\mathrm{d}12}+\left\vert \mathit{\Theta}_{\mathrm{d}%
}(k)\right\vert g_{\mathrm{d}12}\\
&  +\alpha_{\mathrm{d}1}\sup_{m\in\lbrack3hl,3(l+1)h-1]}%
\left\vert \lambda_{\mathrm{d}1}(m)\right\vert +\alpha_{\mathrm{d}%
2}\left\vert w_{\mathrm{d}2}\right\vert _{[3hl,3(l+1)h-1]}.
\end{align*}
By restarting the above estimates from the cycle boundary $3(l-q_{\mathrm{d}})h\in\mathbb{I}_{[k-3(q_{\mathrm{d}}+1)h,k]}$ and using the periodicity of the controller, the same estimates apply to this moving interval. From (\ref{delta_dd}) and the above two estimates, we can get, for $l\geq q_{\mathrm{d}}$,
\begin{align}
&\left\vert x(k)\right\vert
\leq\rho_{\mathrm{d}}(q_{\mathrm{d}})
\left\vert x_{[k-3(q_{\mathrm{d}}+1)h,k]}\right\vert
\nonumber\\
&\quad+\delta_{\mathrm{d}}\left({
\left\vert w_{\mathrm{d}1}\right\vert _{[k-3(q_{\mathrm{d}}+1)h,k]}
+\left\vert w_{\mathrm{d}2}\right\vert _{[k-3(q_{\mathrm{d}}+1)h,k]}}\right).
\end{align}
Here, denote
$V(k)=x(k+3(q_{\mathrm{d}}+1)h)$. Similar to the continuous-time case, according to
$\rho_{\mathrm{d}}(q_{\mathrm{d}})\in(0,1)$ and the discrete-time analogue of Lemma 1
in \cite{Mazen14tac} (see also Theorem 2 in \cite{Mazenc24css}), we can obtain that
\begin{align}
&\left\vert V(k)\right\vert
\leq\left\vert V\right\vert_[-3(q_{\mathrm{d}}+1)h,0]
\mathrm{e}^{\frac{\ln\left(
\rho_{\mathrm{d}}(q_{\mathrm{d}})\right)k}
3(q_{\mathrm{d}}+1)h}
\nonumber\\
&\quad+\frac{\delta_{\mathrm{d}}\left(
\left\vert w_{\mathrm{d}1}\right\vert _{[0,k+3(q_{\mathrm{d}}+1)h]}
+\left\vert w_{\mathrm{d}2}\right\vert _{[0,k+3(q_{\mathrm{d}}+1)h]}\right)}
{\left(1-\rho_{\mathrm{d}}(q_{\mathrm{d}})\right)^{2}}.
\end{align}
for all $k\in\mathbf{%
\mathbb{N}
}$. Therefore, (\ref{x_k_1}) holds. The remainder is similar to the continuous-time case and is therefore omitted.

%
%
%

\end{spacing}
\end{document}